\documentclass[11pt]{amsart}
\usepackage{amsfonts,amscd}
\usepackage{amssymb}
\usepackage{url}
\usepackage{graphicx}
\usepackage[english]{babel}
\theoremstyle{plain}
\newtheorem{theorem}                {Theorem}      [section]
\newtheorem{proposition}  [theorem]  {Proposition}

\newtheorem{lemma}        [theorem]  {Lemma}

\theoremstyle{definition}

\newtheorem{remark}       [theorem]  {Remark}

\numberwithin{equation}{section}

\def \R{{\mathbb R}}
\def \s{{\mathbb S}}

\def \H{{\mathbb H}}
\def \link {~}

\def \b{\mathcal B}
\usepackage{color}

\DeclareMathOperator{\trace}{trace}
\numberwithin{equation}{section}

\begin{document}

\title[Existence and stability of weak critical points]{Existence and stability of weak critical points of $r$-energy functionals}

\author{S.~Montaldo}
\address{Universit\`a degli Studi di Cagliari\\
Dipartimento di Matematica e Informatica\\
Via Ospedale 72\\
09124 Cagliari, Italia}
\email{montaldo@unica.it}

\author{A.~Ratto}
\address{Universit\`a degli Studi di Cagliari\\
Dipartimento di Matematica e Informatica\\
Via Ospedale 72\\
09124 Cagliari, Italia}
\email{rattoa@unica.it}

\author{A.~Sanna}
\address{Universit\`a degli Studi di Cagliari\\
Dipartimento di Matematica e Informatica\\
Via Ospedale 72\\
09124 Cagliari, Italia}
\email{antonio.sanna4@unica.it}

\begin{abstract}
The main aim of this paper is to prove the existence of certain proper weakly $r$-harmonic ($ES-r$-harmonic) maps. We construct critical points which belong to a family of rotationally symmetric maps $\varphi_a : B^n \to \s^n$, where $B^n$ and $\s^n$ denote the Euclidean $n$-dimensional unit ball and sphere respectively. We find that the existence of solutions within this family is restricted to specific dimensions $n$.  Next, we prove that our critical points are \textit{unstable}. In the course of this analysis we point out some specific differences between the $r$-harmonic and the $ES-r$-harmonic cases when $r \geq 4$. Next, we analyse two variants of the problem. First, we replace the target manifold $\s^n$ with a rotationally symmetric ellipsoid $E^n(b)$ and establish the existence of proper weakly biharmonic maps for all $n \geq 5$, as well as proper weakly triharmonic maps for all $n \geq 7$. Finally, we study a similar problem replacing the domain $B^n$ with a suitable warped product manifold.
\end{abstract}



\thanks{The authors are members of the Italian National Group G.N.S.A.G.A. of INdAM. The work was partially supported by the Project {PRoBIKI} funded by Fondazione di Sardegna.  The author A.S. was supported by a NRRP scholarship - funded by the European Union - NextGenerationEU - Mission 4, Component 1, 
Investment 3.4.}

\maketitle

\section{Introduction}\label{intro}

{\it Harmonic maps} are the critical points of the {\em energy} functional
\begin{equation}\label{energia}
E(\varphi)=\frac{1}{2}\int_{M}\,|d\varphi|^2\,dV\, ,
\end{equation}
where $\varphi:M\to N$ is a smooth map between two Riemannian
manifolds $(M,g_M)$ and $(N,g_N)$ of dimension $m$ and $n$ respectively. The condition of harmonicity is equivalent to the fact that the map $\varphi$ is a solution of the Euler-Lagrange equation associated to the energy functional \eqref{energia}, i.e.,
\begin{equation}\label{harmonicityequation}
  - d^* d \varphi =   {\trace}_{g_{_M}} \, \nabla d \varphi =0 \, .
\end{equation}
The left member of \eqref{harmonicityequation} is a vector field along the map $\varphi$ or, equivalently, a section of the pull-back bundle $\varphi^{-1} \, (TN)$: it is called {\em tension field} and denoted $\tau (\varphi)$ (we refer to \cite{EL83, EL1} for background on harmonic maps). 
Now, let us denote $\nabla^M, \nabla^N$ and $\nabla^{\varphi}$ the induced connections on the bundles $TM, TN$ and $\varphi ^{-1}TN$ respectively. The \textit{rough Laplacian} on sections of $\varphi^{-1}  TN$, denoted $\overline{\Delta}$, is defined by
\begin{equation*}
    \overline{\Delta}=d^* d =-\sum_{i=1}^m\left(\nabla^{\varphi}_{e_i}
    \nabla^{\varphi}_{e_i}-\nabla^{\varphi}_
    {\nabla^M_{e_i}e_i}\right)\,,
\end{equation*}
where $\{e_i\}_{i=1}^m$ is a local orthonormal frame field tangent to $M$.
A topic of growing interest deals with the study of the so-called {\it polyharmonic maps}. In this context, it is important to point out that, despite some confusion in the literature, there is a significant difference between {\it $ES-r$-harmonic maps} and $r$-harmonic maps. This fact was studied in detail in \cite{BMOR1}, a rather extensive paper to which we refer for notation and background. 

As proposed in \cite{EL83} and \cite{ES}, {\it $ES-r$-harmonic maps} are the critical points of the $ES-r$-energy functional defined as follows:
\begin{equation}\label{bienergia}
    E^{ES}_r(\varphi)=\frac{1}{2}\int_{M}\,|(d^*+d)^r (\varphi)|^2\,dV\, .
\end{equation}
In the case that $r=2$, the functional \eqref{bienergia} is called bienergy and its critical points are the so-called \textit{biharmonic maps}. There have been extensive studies on biharmonic maps (see \cite{Jiang, SMCO} for an introduction to this topic). 

In the literature, another higher order version of the energy functional, denoted $E_r(\varphi)$, has been studied in several papers (for instance, see \cite{Maeta1, Maeta2, Maeta3, Maeta4, Mont-Ratto1, Mont-Ratto2, Na-Ura, Wang, Wang2}). If $r=2s$, $s \geq 1$, its definition is:
\begin{equation}\label{2s-energia}
\begin{split}
E_{2s}(\varphi)
= \frac{1}{2} \int_M \, \langle \,\overline{\Delta}^{s-1}\tau(\varphi), \,\overline{\Delta}^{s-1}\tau(\varphi)\,\rangle \, \,dV\,.
\end{split}
\end{equation}
In the case that $r=2s+1$, $s\geq 1$:
\begin{equation}\label{2s+1-energia}
\begin{split}
E_{2s+1}(\varphi)
= \frac{1}{2} \int_M \,\sum_{j=1}^m \langle\,\nabla^\varphi_{e_j}\, \overline{\Delta}^{s-1}\tau(\varphi), \,\nabla^\varphi_{e_j}\,\overline{\Delta}^{s-1}\tau(\varphi)\, \rangle \, \,dV \,.
\end{split}
\end{equation}
We say that a map $\varphi$ is \textit{$r$-harmonic} if it is a critical point of the functional $E_r(\varphi)$ defined in \eqref{2s-energia}, \eqref{2s+1-energia}. 
In 1989 Wang \cite{Wang} studied the first variational formula of the $r$-energy functionals $E_r(\varphi)$ and obtained a complete description of the corresponding Euler-Lagrange equations. Moreover, the expression for their second variation was derived in \cite{Maeta3}, where it was shown that a biharmonic map is not always $r$-harmonic ($r \geq 3$) and, more generally, that an $s$-harmonic map is not always $r$-harmonic ($s < r$). On the other hand, any harmonic map is trivially $r$-harmonic and $ES-r$-harmonic for all $r\geq 2$. Therefore, we say that an $r$-harmonic ($ES-r$-harmonic) map is {\it proper} if it is \textit{not} harmonic.
The functionals $E^{ES}_r(\varphi)$ and $E_r(\varphi)$ coincide in the following cases:
\begin{itemize}
\item [{\rm (i)}] $r=2,3$;
\item [{\rm (ii)}] $\dim M=1$;
\item [{\rm (iii)}] the Riemannian curvature tensor of $N$ vanishes.
\end{itemize}
Although for $r > 3$ the critical points of the $ES-r$-energy and of the $r$-energy do not always coincide, the following criterion was established in \cite[p. 370]{MOR22}:
\begin{proposition}\label{pro-general-criterium-ES-r}
Let $r\geq 4$ and  $\varphi:(M,g_M)\to (N,g_N)$ be a smooth map such that 
\[
d^2\overline{\Delta}^i\tau(\varphi) = 0, \quad i=0,\dots,r-4.
\]
Then,  $\varphi$ is $ES-r$-harmonic if and only if it is $r$-harmonic.
\end{proposition}
We also point out that, by contrast with the case of $E_r(\varphi)$, the explicit derivation of the Euler-Lagrange equation for the Eells-Sampson functionals $E^{ES}_r(\varphi)$ seems, in general, a very complicated task. These difficulties are illustrated in \cite{BMOR1}, where the Euler-Lagrange equation of the functional $E^{ES}_4(\varphi)$ was computed, and differences and common features of $E^{ES}_r(\varphi)$ and $E_r(\varphi)$ were thoroughly discussed.

In this paper we first focus on a specific family of rotationally symmetric maps $\varphi_a : B^n \to \s^n$, where $B^n$ and $\s^n$ denote the Euclidean $n$-dimensional unit ball and sphere respectively. More precisely, we define:
\begin{equation}\label{symmetricmaps}\begin{array}{lcll}
                                           \varphi_{a} \,:\, &B^n &\to &  \s^n \subset \R^n \times \R \\
                                           &&& \\
                                           &x &\mapsto & ( \sin a\,{\frac{x}{\rho}}, \,\cos a) \, ,
                                         \end{array}
\end{equation}
where $\rho=|x|$ and $a$ is a constant in the real interval $(0,\pi / 2)$. Of course, $\varphi_{a}$ is well-defined and smooth away from the origin $O$. In Lemma\link\ref{lemma-belong-Sobolev-space} we prove that maps $\varphi_a$ as in \eqref{symmetricmaps} belong to the Sobolev space $W^{r,2}\left ( B^n,\s^n\right )$ if and only if $n \geq 2r+1$ and so, in these cases, they provide natural candidates for weak critical points of energy functionals of order $r$. 

We note that, if one sets $a= \pi /2$ in \eqref{symmetricmaps}, then he would obtain the well-known and widely studied \textit{equator map} (see \cite{FMR2, JK}). The equator map is \textit{weakly harmonic} if $n \geq3$, and so it is a stable, \textit{not proper}, weakly $r$-harmonic and $ES-r$-harmonic map whenever $4\leq 2r+1\leq n$. Thus, since we are interested in \textit{proper} solutions, we restrict our attention throughout the paper to the case $0< a < \pi/2$.

In the biharmonic case existence and stability of weak solutions of the type \eqref{symmetricmaps} was investigated in \cite{FMR}, where we proved the following results:
\begin{theorem}\label{theorem-m=5-6biharmonic} \cite{FMR} Let
$\varphi_a:B^n \to \s^n$ be a map as in \eqref{symmetricmaps}. Then $\varphi_a$ is a proper weakly biharmonic map if and only if either
\begin{itemize}
\item[(i)] $n=5$ and $a=a_2=\pi \slash 3$; or
\item[(ii)] $n=6$ and $a=(1\slash 2)\,\arccos (-4\slash 5)$.
\end{itemize}
\end{theorem}
\begin{theorem}\label{stability-theorem} \cite{FMR} Let
$\varphi_a$ be one of the two proper weakly biharmonic maps of Theorem~\ref{theorem-m=5-6biharmonic}. Then $\varphi_a$ is unstable.
\end{theorem}
As for additional examples of weakly biharmonic maps from $B^n$ to $\mathbb S^n$, the interested reader is referred to \cite{Branding2025}. The above results of \cite{FMR} were obtained by computing with respect to standard Cartesian coordinates on $B^n$ and $\R^{n+1}$. One of the main goals of the present paper is to extend these results to higher orders $r \geq 3$. In this context the required computations are much longer and we found it useful to perform part of them in an intrinsic way, using the identification of $B^n$ and $\s^n$ with suitable models in the sense of \cite{GW}. 
\vspace{2mm}

As a preliminary result we prove:

\begin{proposition}\label{prop-equiv-r-harmonic-es-harmonic}
Assume $r\geq 2$ and $n\geq 2r+1$. Let $\varphi_a:B^n \to \s^n$ be a map as in \eqref{symmetricmaps}. Then $\varphi_a$ is weakly $r$-harmonic if and only if it is weakly $ES-r$-harmonic.
\end{proposition}

In the next result we determine explicitly the critical points of type $\varphi_a:B^n \to \s^n$ as in \eqref{symmetricmaps}. Due to the  high computational complexity we restrict our attention to the cases $3\leq r\leq 5$. 

\begin{theorem}\label{th-constant-soltz} 
Let
$\varphi_a:B^n \to \s^n$ be a map as in \eqref{symmetricmaps}. Then 
\begin{itemize}
\item[(i)] $\varphi_a$ is a proper weakly $ES-3$-harmonic ($3$-harmonic) map if and only if
$$
n=7 \quad {\rm and} \quad a=a_3=\dfrac{1}{2} \arccos\left(\dfrac{1}{9} \left(2\,\sqrt{10}-11\right)\right) \,;
$$
\item[(ii)] $\varphi_a$ is a proper weakly $ES-4$-harmonic ($4$-harmonic) map if and only if
$$
n=9 \quad {\rm and} \quad a=a_4=\dfrac{1}{2} \arccos\left(\dfrac{1}{16} \left(\sqrt{105}-19\right)\right) \,;
$$
\item[(iii)] $\varphi_a$ is a proper weakly $ES-5$-harmonic ($5$-harmonic) map if and only if
$$
n=11 \quad {\rm and} \quad a=a_{5}=\dfrac{1}{2} \arccos\left(\dfrac{1}{25} \left(6\sqrt{6}-29\right)\right) \,.
$$
\end{itemize}
\end{theorem}
In Proposition~\ref{prop-equiv-r-harmonic-es-harmonic} we showed that any map $\varphi_a:B^n \to \s^n$  as in \eqref{symmetricmaps} is 
weakly $r$-harmonic if and only if it is weakly $ES-r$-harmonic. Despite this, if $r\geq 4$ the functionals  $E_r(\varphi)$ and $E^{ES}_r(\varphi)$ do not coincide. Therefore, in the study of the stability of these critical points we will have to take into account the differences between $E^{ES}_r(\varphi)$ and $E_r(\varphi)$ when $r=4,5$. But, as a matter of fact, our instability result is:
\begin{theorem}\label{r-stability-theorem} Let
$\varphi_a:B^n \to \s^n$ be any of the proper weak solutions given in Theorem~\ref{th-constant-soltz}. Then $\varphi_a$ is unstable for both $E_r(\varphi)$ and $E^{ES}_r(\varphi)$.
\end{theorem}
Next, we replace the target manifold $\s^n$ with a rotationally symmetric ellipsoid 
\[
E^n(b)=\left \{( x,y)\in \R^n \times \R \colon |x|^2+\frac{y^2}{b^2}=1\right \}\,.
\]
More precisely, let
\begin{equation}\label{symmetricmaps-ellipsoid}\begin{array}{lcll}
                                           \varphi_{a} \,:\, &B^n &\to &  E^n(b) \subset \R^n \times \R \\
                                           &&& \\
                                           &x &\mapsto & ( \sin a\,{\frac{x}{\rho}}, b\,\cos a) \, ,
                                         \end{array}
\end{equation}
where again $\rho=|x|$ and $0< a < \pi/2$. Our result in the biharmonic case is:
\begin{theorem}\label{th-ellipsoid-biharm}
Assume  $n\geq 5$.  Then there exists a proper weakly biharmonic map $\varphi_a:B^n \to E^n(b)$ of type \eqref{symmetricmaps-ellipsoid} if and only if 
\begin{equation}\label{eq-cond-bihar-ellipsoid}
0<b^2<\frac{(n-1)}{2(n-4)}\,.
\end{equation} 
\end{theorem}
\begin{remark}
We point out that the dimensional restrictions which appear in Theorem\link\ref{theorem-m=5-6biharmonic} do not apply for a suitable choice of $b>0$. Indeed, for any given $n \geq 5$, there exists $b>0$ such that \eqref{eq-cond-bihar-ellipsoid} is satisfied, but $b=1$, i.e., the spherical case, is admissible only when $n=5,6$.
\end{remark}

As for the triharmonic case, we now prove:
\begin{theorem}\label{th-ellipsoid-triharm}
Assume  $n\geq 7$.  Then there exists a proper weakly triharmonic map $\varphi_a:B^n \to E^n(b)$ of type \eqref{symmetricmaps-ellipsoid} if and only if
\begin{equation}\label{eq-cond-trihar-ellipsoid}
0<b^2<\frac{(n-1)}{4(n-6)}\,.
\end{equation} 
\end{theorem}
\begin{remark}
Again, comparing with Theorem~\ref{th-constant-soltz}, we observe that for any given $n \geq 7$ there exists $b>0$ such that \eqref{eq-cond-trihar-ellipsoid} is satisfied, but $b=1$, i.e., the spherical case, is admissible only when $n=7$. 
\end{remark}
Our work is organized as follows: in Section~\ref{preliminaries} we recall some basic facts about Sobolev spaces, weak solutions and stability. In Section~\ref{proofs} we prove Proposition~\ref{prop-equiv-r-harmonic-es-harmonic} and Theorem~\ref{th-constant-soltz}. In Section~\ref{stability} we study the stability of these critical points and prove Theorem~\ref{r-stability-theorem}. In Section~\ref{sec-ellipsoid} we provide the proofs of our results concerning the ellipsoidal target, i.e., Theorems~\ref{th-ellipsoid-biharm} and \ref{th-ellipsoid-triharm}. Finally, in Section~\ref{sec-domain-not-Bn} we study a variant of the problem replacing the domain $B^n$ with a suitable warped product manifold $\b^n_f$. We prove that the existence of biharmonic or triharmonic critical points in our family essentially forces $\b^n_f=B^n$.

\section{Preliminaries}\label{preliminaries}
First, we introduce the most convenient setting to study maps of the type \eqref{symmetricmaps}. Let $(M,g_M)$ be an $m$-dimensional compact Riemannian manifold with boundary $\partial M$ and $\varphi : M \to \s^n$. We consider the canonical embedding $i: \s^n  \hookrightarrow \R^{n+1}$ and write $\varphi=\left (\varphi_1,\ldots,\varphi_{n+1}\right )$ for $i \circ \varphi$. We use the following notation:
\begin{equation}\label{gradiente-Laplaciano}
\nabla \varphi= \left (\nabla \varphi_1,\ldots, \nabla \varphi_{n+1} \right )\quad {\rm and}
\quad \Delta \varphi= \left (\Delta \varphi_1,\ldots, \Delta \varphi_{n+1} \right )\,,
\end{equation}
where $\nabla$ is the gradient on $(M,g_M)$ and so each entry of $\nabla \varphi$ is an $m$-dimensional vector. The Laplacian $\Delta$ acts on a function $u:M\to \R$ as follows: 
\begin{equation}\label{Laplacian-local-coordinates}
 \Delta u = -\,\frac{1}{\sqrt{|g_M|}} \, \frac{\partial}{\partial x_i}\left(\sqrt{|g_M|}\,(g_M)^{ij}\, \frac{\partial u}{\partial x_j} \right )
\end{equation}
(note that in \cite{FMR, Saliba} the opposite sign convention for $\Delta$ was adopted).
Next, let $p$ denote a positive integer. In this context, we recall (see \cite{Aubin, Hebey}) the definition of the following Sobolev spaces: 
\begin{equation}\label{Sobolev-space}
\begin{split}
&W^{p,2}\left ( M,\s^n \right )=\\
&\left \{ \varphi \in W^{p,2}\left ( M,\R^{n+1} \right )\,: \, \varphi(x)=\left (\varphi_1(x),\ldots,\varphi_{n+1}(x)\right ) \in \s^n \,{\rm a.e.} \right \} \,.
\end{split}
\end{equation}
The energy functional \eqref{energia} becomes
\begin{equation}\label{energia-bis}
E(\varphi)=\frac{1}{2}\int_{M}\,|\nabla \varphi|^2\,dV
\end{equation}
and its Euler-Lagrange equation \eqref{harmonicityequation} takes the form
\begin{equation}\label{harmonicityequation-bis}
  \Delta \varphi - \left | \nabla \varphi \right |^2 \varphi =0 \, .
\end{equation}
Then a map $\varphi \in W^{1,2}\left ( M,\s^n \right )$ is \textit{weakly} harmonic  if it is a critical point of \eqref{energia-bis} in $W^{1,2}\left ( M,\s^n \right )$, i.e., if it is a solution of \eqref{harmonicityequation-bis} in the sense of distributions. 

A typical class of weakly harmonic maps is that of {\it minimizers} for the energy functional. To explain this notion, let us recall that in general, if $\varphi_0 \in W^{p,2}\left ( M,\s^n \right )$, then we can define
\begin{equation}\label{Sobolev-space-zero}
\begin{split}
&W_{\varphi_ 0}^{p,2}\left ( M,\s^n \right )=\\
&\left \{ \varphi \in W^{p,2}\left ( M,\s^n \right )\,: \, \left .\nabla^k\left(\varphi-\varphi_0 \right)\right |_{\partial M} \equiv 0,\,0 \leq k \leq p-1 \right \} \,,
\end{split}
\end{equation}
where the boundary condition in \eqref{Sobolev-space-zero} is understood in the sense of traces.

Now, we can say that ${\varphi_ 0} \in W^{1,2}\left ( M,\s^n \right )$ is a minimizer for the energy functional if it satisfies
\[
E(\varphi_0) \leq E(\psi) \quad \forall \,\psi \in W_{\varphi_0}^{1,2}\left ( M,\s^n \right )\,.
\]

As for the bienergy functional \eqref{bienergia}, in our context its expression becomes (see \cite{Baird, FS, Saliba})
\begin{equation}\label{bienergia-bis}
    E_2(\varphi)=\frac{1}{2}\int_{M}\,\left (\left |\Delta \varphi \right |^2\,-\,\left |\nabla \varphi \right |^4 \right )\,dV
\end{equation}
and its Euler-Lagrange equation is given by
\begin{equation}\label{biharmonicityequation-bis}
\begin{split}
  &\Delta^2 \varphi +2\,{\rm div} \left ( \left | \nabla \varphi \right |^2\,\nabla \varphi \right ) \\
  &+ \left (\left |\Delta \varphi \right |^2-\Delta \left |\nabla \varphi \right |^2 -2\, \nabla \varphi \cdot  \nabla \Delta \varphi +2\,\left |\nabla \varphi \right |^4\right )\, \varphi =0 \,,
\end{split}
\end{equation}
where the divergence operator ${\rm div}$ is applied to each component and $\cdot$ denotes scalar product in the following sense:
\[
\nabla \varphi \cdot  \nabla \Delta \varphi= \sum_{j=1}^{n+1}\,\nabla \varphi_j \cdot  \nabla \Delta \varphi_j \, .
\]
Next, we say that a map $\varphi \in W^{2,2}\left ( M,\s^n \right )$ is \textit{weakly} biharmonic  if it is a critical point of \eqref{bienergia-bis} in $W^{2,2}\left ( M,\s^n \right )$, i.e., if it is a solution of \eqref{biharmonicityequation-bis} in the sense of distributions. 
Again, a typical class of weakly biharmonic maps is that of {\it minimizers} for the bienergy functional.
Indeed, we say that $\varphi_0 \in W^{2,2}\left ( M,\s^n \right )$ is a minimizer if it satisfies
\[
E_2(\varphi_0) \leq E_2(\psi) \quad \forall \,\psi \in W_{\varphi_0}^{2,2}\left ( M,\s^n \right )\,.
\]
An important step towards understanding whether a given weakly biharmonic map is a minimizer is the study of its \textit{stability}. More precisely, let $\varphi_0 \in W^{2,2}\left ( M,\s^n \right )$ be a weakly biharmonic map and denote by $\varphi_s$ ($s \geq0$) a variation of $\varphi_0$ through maps in $W_{\varphi_0}^{2,2}\left ( M,\s^n \right )$. We say that $\varphi$ is  \textit{stable} if
\begin{equation}\label{second-variation}
 \left . \frac{d^2}{ds^2}\,E_2\left( \varphi_s \right ) \right |_{s=0} \, \geq \, 0
\end{equation}
for all such variations $\varphi_s$. In particular, if $\varphi_0$ is \textit{not} stable, then it cannot be a minimizer. Moreover, we point out that the notions of minimizer and stability extend to a critical point $\varphi_0$ of the $r$-energy functionals $E_r(\varphi)$, $E_r^{ES}(\varphi)$ simply by replacing $W^{2,2}\left ( M,\s^n \right )$ and $W_{\varphi_0}^{2,2}\left ( M,\s^n \right )$ by $W^{r,2}\left ( M,\s^n \right )$ and $W_{\varphi_0}^{r,2}\left ( M,\s^n \right )$ respectively, $r \geq 3$.
\vspace{2mm}

Formally, the Euler-Lagrange equations associated to the functionals $E_r(\varphi)$, $E_r^{ES}(\varphi)$ are elliptic PDE systems of order $2r$ and so we can say that any map $\varphi \in W^{r,2}\left ( B^n,\s^n \right )$ is a weak critical point if it is a solution of the Euler-Lagrange equation in the sense of distributions. Unfortunately, it is rather difficult to obtain an explicit version of \eqref{bienergia-bis} and \eqref{biharmonicityequation-bis} for these functionals. Therefore, in order to overcome this difficulty, we adopt an \textit{intrinsic approach} which works because of the following key observation, which is an obvious consequence of the fact that $\varphi_a$ is smooth away from the measure zero set $\{O\}$. We state it in the form of a proposition in order to make it easier to refer to it:
\begin{proposition}\label{pro:weakly-criterium} Let
$\varphi_a:B^n \to \s^n$ be a map as in \eqref{symmetricmaps}. Then $\varphi_a$ is a weak critical point of $E^{ES}_r(\varphi)$ ($E_r(\varphi)$) if and only if  $\varphi_a \in W^{r,2}\left ( B^n,\s^n \right )$ and $\varphi_a$ is $ES-r$-harmonic ($r$-harmonic) on $B^n \setminus \{ O\}$.
\end{proposition}
By way of summary, in order to prove Theorem\link\ref{th-constant-soltz} it will suffice to understand when $\varphi_a \in W^{r,2}\left ( B^n,\s^n \right )$ and, taking into account Proposition~\ref{prop-equiv-r-harmonic-es-harmonic}, it is $r$-harmonic on $B^n \setminus \{ O\}$. This analysis will be carried out in the next section.
\section{Proofs of Proposition~\ref{prop-equiv-r-harmonic-es-harmonic} and Theorem~\ref{th-constant-soltz}}\label{proofs}
First, it is important to establish when a map $\varphi_a$ as in \eqref{symmetricmaps} belongs to $W^{r,2}\left ( B^n,\s^n \right )$:
\begin{lemma}\label{lemma-belong-Sobolev-space}
Let
$\varphi_a:B^n \to \s^n$ be a map as in \eqref{symmetricmaps}. Then $\varphi_a \in W^{r,2}\left ( B^n,\s^n \right )$ if and only if $n\geq 2r+1$.
\end{lemma}
\begin{proof} We observe that $\varphi_a \in W^{r,2}\left ( B^n,\s^n \right )$ if and only if:\\
when $r$ is even for all $1 \leq k \leq \frac{r}{2}$
\[
\int_{B^n}  \left | \nabla \left (\Delta^{k-1}\varphi_a\right )\right |^2 \, dV < +\infty \quad {\rm and} \quad \int_{B^n}  \left | \Delta^k \varphi_a\right |^2 \, dV < +\infty\,;
\]
when $r$ is odd for all $0 \leq k \leq \frac{r-1}{2}$ 
\[
\int_{B^n} \left | \Delta^k \varphi_a\right |^2 \, dV < +\infty \quad {\rm and}\quad \int_{B^n}  \left | \nabla \left (\Delta^{k}\varphi_a\right )\right |^2 \, dV < +\infty  \,,
\]
where, as usual, $\Delta^0 \varphi=\varphi$. We perform the computation using Lemma 3.1 of \cite{FMR} and, for all $k \geq 0$, we obtain:
\begin{equation}\label{formula-lemma-sobolev}
\begin{split}
\int_{B^n} \, \left | \nabla \left (\Delta^{k}\varphi_a\right )\right |^2 \, dV &= c_1(k,n)\,(\sin^2 a)\,\int_0^1 \,  \rho^{n-4k-3}\,d\rho  \,; \\ 
 \int_{B^n} \, \left | \Delta^k \varphi_a\right |^2 \, dV &= c_2(k,n)\,(\sin^2 a )\,\int_0^1 \,  \rho^{n-4k-1}\,d\rho \,,
\end{split}
\end{equation}
where $c_1(k,n)$ and $c_2(k,n)$ are positive constants which depend on $k$ and $n$. Now, the conclusion of the lemma follows easily from \eqref{formula-lemma-sobolev}.
\end{proof}
Next, we have to study $r$-harmonicity and $ES-r$-harmonicity of $\varphi_a$ on $B^n \backslash \{ O\}$. It is convenient to introduce a more general family of rotationally symmetric maps.
To this purpose, let us first introduce a family of warped product manifolds which will be suitable for our purposes. We set
\begin{equation}\label{rotationallysymmetricmanifolds}
\left( M,g_M \right )= \left (\s^{n-1} \times I, f^2(\rho)g_{\s}+d\rho^2 \right ),
\end{equation}
where $g_{\s}$ is the standard metric on $\s^{n-1}$, $I\subset \R$ is an open interval and $f(\rho)$ is a smooth function which is positive on $I$.
\begin{remark}\label{remark-interior-I}
In some instances, it may be of interest to extend the analysis through the closure $\overline{I}$ of $I$. By way of example, if $\overline{I}=[0,+\infty)$ and
\begin{equation}\label{condizioni-su-f}
\begin{cases}
    f(0)=0 \,, \quad f'(0)=1  ; \\[6pt]
    f^{(2k)}(0)=0 \quad {\rm for} \,\, {\rm all } \,\, k \geq 1 \,\, ,\\
  \end{cases}
\end{equation}
then the manifold \eqref{rotationallysymmetricmanifolds} becomes a \textit{model} in the sense of Greene and Wu (see \cite{GW, Petersen}). In particular, if $f(\rho)=\rho$ (respectively, $f(\rho)=\sinh \rho$) it is isometric to the Euclidean space $\R^{n}$ (respectively, the hyperbolic space $\H^{n}$). In a similar spirit, if $\overline{I}=[0,\pi]$ and $f(\rho)=\sin \rho$, then we have the Euclidean unit sphere $\s^{n}$.
We also consider the case that $M$ has a boundary. For instance, if $f(\rho)=\rho$ and $\overline{I}=[0,1]$, then we have the Euclidean unit ball $B^n$.

We point out that all the calculations and results of this section are valid on $I$. In particular, the study of regularity across the loci associated to $\partial I$ (poles or boundary of $\left( M,g_M \right )$) needs a case by case analysis.
\end{remark}
By way of summary, we refer to a manifold as in \eqref{rotationallysymmetricmanifolds} as to a \textit{rotationally symmetric manifold} and, to shorten notation, we often write $M_{f}$ to denote it. The family of rotationally symmetric maps which we are interested in is the following:
\begin{equation}\label{rotationallysymmetricmaps-models}\begin{array}{lrll}
                                           \varphi_{\alpha} \,:\, &\left (\s^{n-1} \times I, f^2(\rho)g_{\s}+d\rho^2 \right ) &\to &  \left (\s^{n-1} \times I', h^2(\alpha)g_{\s}+d\alpha^2 \right ) \\
                                           &&& \\
                                           &(w,\rho) &\mapsto & (w,\alpha(\rho)) \,\, ,
                                         \end{array}
\end{equation}
where $\alpha(\rho)$ is a smooth function on a real interval $I=(a,b)$ with values in $I'=(a',b')$. To denote a rotationally symmetric map as in \eqref{rotationallysymmetricmaps-models} we write $\varphi_{\alpha}:M_f \to M_h$ or, if the context is clear, simply $\varphi_\alpha$. 

For later use we recall that we locally work using the frame field 
\begin{equation}\label{eq:coordinate-field}
\left\{ \dfrac{\partial}{\partial w_1}, \dots, \dfrac{\partial}{\partial w_{n-1}}, \dfrac{\partial}{\partial w_{n}}=\dfrac{\partial}{\partial \rho}\right\}\,,
\end{equation}
where   $\left (w_1,\dots,w_{n-1}, w_n=\rho \right )$  are local coordinates on $M_f$ with  $\left (w_1,\dots,w_{n-1} \right )$  local coordinates  on $\s^{n-1}$. 
Moreover, their associated Christoffel symbols $\Gamma^k_{ij}$ are described by the following table:
\begin{equation}\label{simb-christ}
\begin{array}{lll}
{\rm (i)}&{\rm If}\, 1 \leq i,j,k \leq n-1: & \Gamma^k_{ij}={}^{\s}\Gamma^k_{ij} \\
{\rm (ii)}&{\rm If}\, 1 \leq i, j \leq n-1: & \Gamma^n_{ij}=\, -\,f(\rho)\,f'(\rho)\,\,(g_{\s})_{ij} \\
{\rm (iii)}&{\rm If}\, 1 \leq i,j \leq n-1: & \Gamma^j_{i n}=\frac{f'(\rho)}{f(\rho)} \,\,\delta_i^j \\
{\rm (iv)}&{\rm If}\, 1 \leq j \leq n: & \Gamma^j_{nn}=0=\Gamma^n_{jn} \,\,,\\
\end{array}
 \end{equation}
where ${}^{\s}\Gamma^k_{ij}$ and $(g_{\s})_{ij}$ denote respectively the Christoffel symbols and the components of metric tensor of $\s^{n-1}$ with respect to the local coordinates $w_1, \ldots,w_{n-1}$.

Now, let $\varphi_a$ be a map as in \eqref{symmetricmaps}. We note that the restriction of $\varphi_a$ to $B^n \backslash \{ O\}$ corresponds to the special case of \eqref{rotationallysymmetricmaps-models} where we choose $f(\rho)=\rho$, $h(\alpha)= \sin \alpha$, $I=(0,1]$, $I'=(0,\pi)$ and $\alpha(\rho)\equiv a$. In \cite{BMOR1} we carried out a detailed study of maps as in \eqref{rotationallysymmetricmaps-models} and we report here the results which are relevant in our context. First, we recall that the tension field of $\varphi_{\alpha}$ is given by:
\begin{equation}\label{tension-field-fi-alfa}
\tau \left (\varphi_{\alpha}\right )= \tau_\alpha (\rho)\, \frac{\partial}{\partial \alpha},
\end{equation}
where
\begin{equation}\label{tension-field-fi-alfa-funzione}
\tau_\alpha =\ddot{\alpha} +(n-1)\, \frac{\dot{f}}{f}\,\dot{\alpha}-\frac{(n-1)}{f^2}\, h(\alpha)\,h'(\alpha)
\end{equation}
and $\dot{\;}$ denotes the derivative with respect to $\rho$. We have the following general result:
\begin{theorem}\label{theorem-r-energy}
\cite{BMOR1} Let us denote
\begin{equation}\label{assumptions-recursivity}
\begin{split}
\mathcal{T}_2&= \tau_{\alpha}  \\
\mathcal{T}_{2k}&= \ddot{\mathcal{T}}_{2(k-1)}+\,(n-1) \,\frac{\dot{f}}{f}\,\dot{\mathcal{T}}_{2(k-1)}\, -\,(n-1)\,\frac{h'^2( \alpha)}{f^2}\,\mathcal{T}_{2(k-1)}  \quad ( k \geq 2)\\ 
\mathcal{T}_{2k+1}&=\left [\dot{\mathcal{T}}_{2k}^2+(n-1) \,\frac{h'^2( \alpha)}{f^2}\, \mathcal{T}_{2k}^2\right ]^{1/2} \quad (k \geq 1)\,\,,
\end{split}
\end{equation}
where $\tau_{\alpha}(\rho)$ is the function introduced in \eqref{tension-field-fi-alfa-funzione} and $\dot{\;}$ indicates the derivative with respect to $\rho$. Then the $r$-energy of a rotationally symmetric map $\varphi_{\alpha}$ as in \eqref{rotationallysymmetricmaps-models} is
\begin{equation}\label{r-energy-rot-symm-maps-Lagrangian}
E_r(\varphi_\alpha)={\rm Vol}(\s^{n-1})\,\int_a^b\,L_r\left(\rho, \alpha(\rho), \dot{\alpha}(\rho),\ldots, \alpha^{(r)}(\rho)\right) \, d\rho \quad (r \geq 2) \,,
\end{equation}
where the explicit expression for the Lagrangians $L_r$ is:
\begin{align}\label{r-energy-explicit-rot-maps}
L_r &=  \frac{1}{2}\,\,\mathcal{T}_{r}^2\, \,f^{n-1}\,  \,\, \quad \quad (r \geq 2) \,.
\end{align}
Moreover, $\varphi_\alpha$ is an $r$-harmonic map if and only if the function $\alpha$ satisfies the Euler-Lagrange equation
\begin{equation}\label{Euler-lagrange-generale}
\sum_{i=1}^r \, (-1)^i \, \frac{d^i}{d\rho^i}\,\left( \frac{\partial L_r}{\partial \alpha^{(i)}}\right)+\frac{\partial L_r}{\partial \alpha}=0 \,\,.
\end{equation}
\end{theorem}
\vspace{3mm}

We also recall the following technical lemma

\begin{lemma}[\cite{BMOR1}]\label{lemma-rough-laplacian-partialt}
	Let $\varphi_{\alpha}$ be a map of type \eqref{rotationallysymmetricmaps-models} and $W = F(\rho) \frac{\partial}{\partial \alpha} \in C(\varphi_{\alpha}^{-1}T M_h)$, where $F$ is a smooth function that depends on $\rho \in I$. Then
	\begin{equation}\label{expr-rough-laplacian-partialt}
		\overline{\Delta} W = -\left(\ddot{F} + (n-1)\frac{\dot{f}}{f}\dot{F}-(n-1)\frac{h'^2(\alpha)}{f^2}F\right)\frac{\partial}{\partial \alpha}.
	\end{equation}
\end{lemma}

Now we are in the right position to prove Proposition~\ref{prop-equiv-r-harmonic-es-harmonic}. Indeed, 

\begin{proof}[{\bf Proof of Proposition~\ref{prop-equiv-r-harmonic-es-harmonic}}]
We first prove that a map of type \eqref{rotationallysymmetricmaps-models} with $\alpha(\rho)=a$ satisfies the criterion of Proposition~\ref{pro-general-criterium-ES-r}, that is
\[
d^2\overline{\Delta}^i\tau(\varphi_{a}) = 0, \quad i=0,\dots,r-4.
\]

Using Lemma~\ref{lemma-rough-laplacian-partialt} it is easy to check by an induction argument that, for $k\geq 0$,
\begin{equation}\label{eq:deltaktau}
\overline{\Delta}^k \tau(\varphi_{\alpha}) = (-1)^k\mathcal{T}_{2(k+1)} \dfrac{\partial}{\partial \alpha}\,.
\end{equation}

Now, we recall that, for $Z \in C(\varphi_{\alpha}^{-1}TM_h)$,

\begin{equation}\label{eq:d2Z}
d^2 Z (X, Y) = R^{M_h}(d\varphi_{\alpha}(X),d\varphi_{\alpha}(Y))Z, \quad \forall X, Y \in C(TM_f).
\end{equation}
Therefore, for a fixed  $k$, using the standard expression of the curvature tensor of a warped product (see \cite[Chapter~7, Proposition~42]{Oneill}), we can compute $d^2 \overline{\Delta}^k \tau(\varphi_{\alpha})$ on the local coordinate frame field \eqref{eq:coordinate-field}.
We obtain, for $1\leq i,j\leq n-1$, 
\begin{equation}\label{eq:d2deltaktau}
\begin{split}
	d^2 \overline{\Delta}^k \tau(\varphi_{\alpha}) \left(\dfrac{\partial}{\partial w_i}, \dfrac{\partial}{\partial w_j}\right) & = (-1)^k\mathcal{T}_{2(k+1)} R^{M_h}\left(\dfrac{\partial}{\partial w_i}, \dfrac{\partial}{\partial w_j}\right) \dfrac{\partial}{\partial \alpha} = 0;\\
	d^2 \overline{\Delta}^k \tau(\varphi_{\alpha}) \left(\dfrac{\partial}{\partial \rho}, \dfrac{\partial}{\partial \rho} \right) & = (-1)^k \dot{\alpha}^2 \mathcal{T}_{2(k+1)} R^{M_h}\left(\dfrac{\partial}{\partial \alpha}, \dfrac{\partial}{\partial \alpha} \right) \dfrac{\partial}{\partial \alpha} = 0;\\
	d^2 \overline{\Delta}^k \tau(\varphi_{\alpha}) \left(\dfrac{\partial}{\partial w_i}, \dfrac{\partial}{\partial \rho} \right) & = (-1)^k \dot{\alpha} \mathcal{T}_{2(k+1)} R^{M_h}\left(\dfrac{\partial}{\partial w_i}, \dfrac{\partial}{\partial \alpha} \right) \dfrac{\partial}{\partial \alpha} \\ &= (-1)^{k+1} \dot{\alpha} \mathcal{T}_{2(k+1)} \dfrac{h''(\alpha)}{h(\alpha)}\dfrac{\partial}{\partial w_i}.\\
\end{split}
\end{equation}
Setting $\alpha(\rho) = a$ we easily conclude that $d^2 \overline{\Delta}^k \tau(\varphi_{\alpha})=0$ for any fixed $k\geq 0$.

Specializing to a map $\varphi_a:B^n \to \s^n$, $n\geq 2r+1$, as in \eqref{symmetricmaps} we obtain that $\varphi_a$ restricted to $B^n\setminus\{O\}$ is $r$-harmonic if and only if it is $ES-r$-harmonic. Next,  combining Proposition~\ref{pro:weakly-criterium} and Lemma~\ref{lemma-belong-Sobolev-space} we conclude that $\varphi_a$ is weakly $r$-harmonic if and only if it is weakly $ES-r$-harmonic.
\end{proof}

\begin{proof}[\textbf{Proof of Theorem~\ref{th-constant-soltz}}]
We prove cases (i) and (iii), since case (ii) was already treated in \cite{BMOR1}.

\smallskip

\noindent
\textbf{Case (i).}
Using \eqref{r-energy-explicit-rot-maps} and \eqref{Euler-lagrange-generale} with $r=3$, $f(\rho)=\rho$, and $h(\alpha(\rho))=\sin(\alpha(\rho))$ we find that a map $\varphi_\alpha$ with $\alpha(\rho)\equiv a\in (0,\pi/2)$ is a proper weakly $3$-harmonic map if and only if $a$ is a nonzero root of the polynomial
\[
3(n-1)^2\cos(4a)
+4(n-1)(9n-41)\cos(2a)
+81n^2-754n+1633.
\]

Setting $x=\cos^2 a$, the problem reduces to finding real roots in the interval $(0,1)$ of the polynomial
\begin{equation}\label{eq-poly-r3}
3(n-1)^2 x^2 + 2(n-1)(3n-19)x + 2(n-4)(3n-23).
\end{equation}

The polynomial \eqref{eq-poly-r3} has real roots if and only if $3\le n\le 8$. Since, by Lemma~\ref{lemma-belong-Sobolev-space}, one must have $n\ge 7$, the only possible cases remain $n=7$ and $n=8$.

Now, if $n=8$ a direct computation shows that both roots are negative. When $n=7$ there is exactly one root in the interval $(0,1)$ which yields the corresponding value
\[
a=a_3=\frac{1}{2}\arccos\!\left(\frac{1}{9}\bigl(2\sqrt{10}-11\bigr)\right).
\]

\smallskip

\noindent
\textbf{Case (iii).}
As for weakly $5$-harmonic maps, a computation analogous to the previous case shows that a map $\varphi_\alpha$ with $\alpha(\rho)\equiv a \in (0,\pi /2)$ is a proper weakly $5$-harmonic map if and only if $x=\cos^2 a$ is a root in $(0,1)$ of the polynomial
\begin{equation}\label{eq-poly-r5}
\big((n-1)x+2n-8\big)P_3(x)
\qquad \text{with } n\ge 11,
\end{equation}
where
\[
\begin{aligned}
P_3(x)={}&5(n-1)^3 x^3
+2(n-1)^2(17n-138)x^2 \\
&+4(n-1)(27n^2-475n+1984)x
+24(n-6)(n-8)(7n-79).
\end{aligned}
\]

First, we observe that when $n\ge 11$ the polynomial $(n-1)x+2n-8$ has only a negative root. Hence, admissible solutions must arise from the roots of $P_3(x)$.

In the case that $n=11$ we have $P_3(0)<0$ and $P_3(1)>0$, thus there must be an admissible root. A direct computation shows that $P_3(x)$ has exactly one positive root, which yields the corresponding value
\[
a=a_5=\frac{1}{2}\arccos\!\left(\frac{1}{25}\bigl(6\sqrt{6}-29\bigr)\right).
\]

If $n>11$, then $P_3(0)>0$ and it is straightforward to verify that for $n>14$ the derivative of $P_3(x)$ is positive for all $x$, so that $P_3$ has no admissible roots in $(0,1)$. If $12 \leq n \leq 14$ a direct inspection shows that the unique real root of $P_3(x)$ is negative.

Therefore, the polynomial in \eqref{eq-poly-r5} admits a root in $(0,1)$ if and only if $n=11$.
\end{proof}

\begin{remark}\label{rem-conjecture}
We have observed that the weak $r$-harmonic critical points $\varphi_{a_r}$ of Theorems~\ref{theorem-m=5-6biharmonic}(i) and \ref{th-constant-soltz}  can be described by means of the following pattern:
\[
n=2 r +1\,, \quad {\rm with}\quad a_r=\frac{1}{2} \arccos \left(\frac{\sqrt{\left(r^2-1\right) (2 r-1)}-r^2-r+1}{r^2}\right)\,.
\]
We conjecture that these weakly $r$-harmonic maps do exist for all $r \geq 2$. A computation carried out using the Wolfram software \textit{Mathematica}${}^{ \text{\textregistered}} $, based on the explicit case by case calculation of the Euler-Lagrange equations according to Theorem~\ref{Euler-lagrange-generale}, enabled us to verify  that \textit{the conjecture is true for all }$2 \leq r \leq 8$. A complete proof for $r \geq 9$ seems technically very demanding because it is difficult to establish a recursive law to control $L_r$ and its associated Euler-Lagrange equations as $r$ increases.
\end{remark}

\section{Stability}\label{stability}

In Proposition~\ref{prop-equiv-r-harmonic-es-harmonic} we showed that any map $\varphi_a:B^n \to \s^n$  as in \eqref{symmetricmaps} is 
weakly $r$-harmonic if and only if it is weakly $ES-r$-harmonic. Nevertheless, if $r\geq 4$ the functionals $E_r(\varphi)$ and $E^{ES}_r(\varphi)$ do not coincide and so the study of the stability of the critical points in Theorem~\ref{th-constant-soltz} (ii), (iii) also requires the explicit computations of both $E^{ES}_4(\varphi)$ and $E^{ES}_5(\varphi)$. 

When $r=4$, we found in \cite{BMOR1}:
\begin{equation} \label{4-ES-energy-rot-symm-maps}
E^{ES}_4(\varphi_\alpha)= \frac{1}{2}\,{\rm Vol}(\s^{n-1})\,\int_a^b\,\left [(n-1)\,\dot{\alpha}^2 \,\tau^2_\alpha\, \,\frac{h''^2 (\alpha)}{f^2} \right] \, f^{n-1}\,d\rho +E_4(\varphi_\alpha)\,.
\end{equation}

No recursive method, such as in Theorem~\ref{theorem-r-energy} for instance, is available to compute $E_r^{ES}(\varphi)$ when $r \geq 5$. Thus we have to perform the explicit tedious computation of $E^{ES}_5(\varphi)$. The following result will be useful for our study of stability, but it may be considered of independent interest for future studies of rotationally symmetric $ES-5$-harmonic maps between models.

\begin{theorem}\label{th-Euler-Lag-r=5} Let $\varphi_\alpha:M_f \to M_h$ be 
a rotationally symmetric map as in \eqref{rotationallysymmetricmaps-models}.
Then
\[
E^{ES}_5(\varphi)=\frac{1}{2}\,{\rm Vol}(\s^{n-1})\,\int_a^b\, L_5^{ES} d\rho\,,
\]
with
{\footnotesize
\begin{equation}\label{eq:ES-5-energy}
\begin{split}
&L_5^{ES}=\\
&\frac{1}{2}  {\Large\Big\{}\frac{(n-1)}{f^2} { \left[ h(\alpha) \frac{d}{d \rho}\left(\dot{\alpha} \tau_{\alpha} \frac{h''(\alpha)}{h(\alpha)}\right) +  \dot{\alpha}^2 \tau_{\alpha} \frac{h'(\alpha)h''(\alpha)}{h(\alpha)}+ (n-3) \dot{\alpha} \tau_{\alpha} h''(\alpha)\frac{\dot{f}}{f} \right]^2}\\&
+{\frac{(n-1)}{f^2}} {\dot\alpha}^2\, \tau_{\alpha}  \left[{(n-1)}\tau_{\alpha} \frac{{h'}^2(\alpha ) {h''}^2(\alpha )}{f^2} - {2 \mathcal{T}_{4}\, {h''}^2(\alpha )}\right]{\Large\Big\}}f^{n-1}+L_5
\end{split}
\end{equation}
}
where, according to \eqref{assumptions-recursivity} and \eqref{r-energy-explicit-rot-maps}, 
\[
\begin{split}
\mathcal{T}_{4}=&\ddot\tau_{\alpha} + (n-1)\dot\tau_{\alpha}\frac{ \dot f }{f} -(n-1) \tau_{\alpha}\frac{ {h'}^2(\alpha)}{f^2}\,,\\
L_5=&\frac{1}{2} \left\{\frac{(n-1)}{f^2}{\mathcal{T}_{4}^2\, {h'}^2(\alpha)}+\dot{\mathcal{T}_{4}}^2\right\} f^{n-1}\,.
\end{split}
\]
\end{theorem}
\begin{proof}
We start computing the $ES-5$-energy functional for a smooth map  $\varphi: (M,g) \to (N,h)$ between two Riemannian manifolds.
Taking into account that the codifferential operator $d^*$ vanishes when it is applied to $0$-form and $d^2 \varphi = 0$, we have 
\[
(d^*+d)^5 (\varphi) = -(d^*+d) (\overline{\Delta}\tau(\varphi) + d^2 \tau(\varphi)) = -(d\overline{\Delta}\tau(\varphi) + d^*d^2 \tau(\varphi)+d^3 \tau(\varphi)).
\]

Then we compute the norm squared and, since the scalar product between forms with different degree does not contribute, we obtain

\[
|(d^*+d)^5 (\varphi)|^2 = |d\overline{\Delta}\tau(\varphi)|^2 + |d^*d^2 \tau(\varphi)|^2 + |d^3 \tau(\varphi)|^2 + 2\langle d\overline{\Delta}\tau(\varphi), d^*d^2 \tau(\varphi)\rangle.
\]

Therefore, using the fact that $d^*$ is the adjoint of $d$, we conclude that
\begin{equation}\label{ES-5-energia-generale}
\begin{split}
E^{ES}_5(\varphi)= &E_5(\varphi) + \frac{1}{2}\int_{M}|d^*d^2 \tau(\varphi)|^2 dV + \frac{1}{2}\int_{M} |d^3 \tau(\varphi)|^2 dV\\
 &+ \int_{M}\langle d^2\overline{\Delta}\tau(\varphi), d^2 \tau(\varphi)\rangle dV.
\end{split}
\end{equation}
Now we compute the last three terms in the right-hand side of \eqref{ES-5-energia-generale} assuming that the map $\varphi$ is a rotationally symmetric map $\varphi_{\alpha}$ as in \eqref{rotationallysymmetricmaps-models}.
With respect to the local coordinate frame field \eqref{eq:coordinate-field}
 we have
\begin{equation}\label{expr-codif-d2tau-normsquared}
\begin{split}
|d^*d^2 \tau(\varphi_\alpha)|^2 = & \sum_{i,j=1}^{n-1}\frac{(g_\mathbb{S})^{ij}}{f^2}\langle d^*d^2 \tau(\varphi_\alpha) \left({\frac{\partial}{\partial w_i}}\right), d^*d^2 \tau(\varphi_\alpha)\left({\frac{\partial}{\partial w_j}}\right)\rangle\\
& + \langle d^*d^2 \tau(\varphi_\alpha)\left({\frac{\partial}{\partial \rho}}\right), d^*d^2 \tau(\varphi_\alpha)\left({\frac{\partial}{\partial \rho}}\right) \rangle. 
\end{split}
\end{equation}

By the definition of $d^*$, the $1$-form $d^*d^2 \tau(\varphi_\alpha)$ evaluated on $X \in C(TM_f)$ is (see for example \cite[pag. 8]{EL83}):
\begin{align*}
d^*d^2 \tau(\varphi_\alpha)(X)  = &- \sum_{k,\ell =1}^n (g_{M_f})^{k \ell} \left(\nabla_{\frac{\partial}{\partial w_\ell }} d^2 \tau(\varphi_\alpha)\right)\left( \frac{\partial}{\partial w_k}, X\right)\\
= & - \sum_{k,\ell =1}^n (g_{M_f})^{k\ell } \left\{\nabla^{\varphi_{\alpha}}_{\frac{\partial}{\partial w_\ell }} \left(d^2 \tau(\varphi_\alpha)\left( \frac{\partial}{\partial w_k}, X\right) \right)\right.  \\
& \left.  - d^2 \tau(\varphi_\alpha)\left( \nabla^{M_f}_{\frac{\partial}{\partial w_\ell }}\frac{\partial}{\partial w_k}, X\right) -  d^2 \tau(\varphi_\alpha)\left( \frac{\partial}{\partial w_k}, \nabla^{M_f}_{\frac{\partial}{\partial w_\ell }}X\right) \right\} \\
\end{align*}
which, using the expression of the Christoffel symbols for a model given in \eqref{simb-christ}, becomes
\begin{align}\label{expr-codif-d2tau}
d^*d^2& \tau(\varphi_\alpha)(X)  =\nonumber\\
 &- \sum_{i,j=1}^{n-1 }\frac{(g_\mathbb{S})^{ij}}{f^2} \left\{\nabla^{\varphi_{\alpha}}_{\frac{\partial}{\partial w_j}} \left(d^2 \tau(\varphi_\alpha)\left(\frac{\partial}{\partial w_i}, X\right)\right) - {}^\mathbb{S}\Gamma_{ji}^k d^2 \tau(\varphi_\alpha)\left( \frac{\partial}{\partial w_k}, X \right) \right. \nonumber \\
& \left.  + (g_\mathbb{S})_{ji} f\dot{f}  d^2 \tau(\varphi_\alpha)\left( \frac{\partial}{\partial \rho}, X \right)  -  d^2 \tau(\varphi_\alpha)\left(\frac{\partial}{\partial w_i}, \nabla^{M_f}_{\frac{\partial}{\partial w_j}}X\right) \right\}\\
& -\nabla^{\varphi_{\alpha}}_{\frac{\partial}{\partial \rho}} d^2 \tau(\varphi_\alpha)\left(\frac{\partial}{\partial \rho}, X \right) +  d^2 \tau(\varphi_\alpha)\left(\frac{\partial}{\partial \rho}, \nabla^{M_f}_{\frac{\partial}{\partial \rho}}X\right)\,. \nonumber
\end{align}	

Next, substituting $X = \frac{\partial}{\partial w_\ell}$, $\ell=1, \dots, n-1$ in \eqref{expr-codif-d2tau} and using \eqref{simb-christ}, \eqref{eq:d2deltaktau}, after a standard computation we obtain: 
\begin{align*}
&d^*d^2 \tau(\varphi_\alpha)\left(\frac{\partial}{\partial w_\ell}\right) 	
= \\
&- \left( \frac{d}{d \rho}\left(\dot{\alpha} \tau_{\alpha} \frac{h''(\alpha)}{h(\alpha)}\right) +  \dot{\alpha}^2 \tau_{\alpha} \frac{h'(\alpha)h''(\alpha)}{h^2(\alpha)} + (n-3) \dot{\alpha} \tau_{\alpha} \frac{h''(\alpha)}{h(\alpha)}\frac{\dot{f}}{f} \right)\frac{\partial}{\partial w_\ell}  \,.
\end{align*}	

Similarly, if we choose $X = \frac{\partial}{\partial \rho}$, then \eqref{expr-codif-d2tau} becomes
\begin{align*}
d^*d^2 \tau(\varphi_\alpha)\left(\frac{\partial}{\partial \rho}\right) 	
= & -(n-1) \dot{\alpha} \tau_{\alpha} \frac{h'(\alpha)h''(\alpha)}{f^2} \frac{\partial}{\partial \alpha}\,.
\end{align*}	

Thus, substituting these two expressions in \eqref{expr-codif-d2tau-normsquared}, we conclude that

{\small
\begin{align}\label{expr-codif-d2tau-normsquared-final}
|d^*d^2 &\tau(\varphi_\alpha)|^2 =\nonumber\\
 & \frac{n-1}{f^2} \left[ h(\alpha) \frac{d}{d \rho}\left(\dot{\alpha} \tau_{\alpha} \frac{h''(\alpha)}{h(\alpha)}\right) +  \dot{\alpha}^2 \tau_{\alpha} \frac{h'(\alpha)h''(\alpha)}{h(\alpha)} + (n-3) \dot{\alpha} \tau_{\alpha} h''(\alpha)\frac{\dot{f}}{f} \right]^2 \nonumber\\
& + (n-1)^2 \dot{\alpha}^2 \tau_{\alpha}^2 \left(\frac{h'(\alpha)h''(\alpha)}{f^2}\right)^2.
\end{align}
}
The next step is to compute the scalar product between the $2$-forms $d^2\overline{\Delta}\tau(\varphi)$ and $d^2 \tau(\varphi)$. 

From \eqref{eq:d2deltaktau} we know that the only nonzero components of the two $2$-forms, when applied to pairs of vector fields from the coordinate frame  \eqref{eq:coordinate-field}, are
\begin{equation}\label{eq:nonzeroinnerproductd2deltad2}
\begin{split}
d^2\overline{\Delta}\tau(\varphi)\left(\frac{\partial}{\partial w_i},\frac{\partial}{\partial \rho}\right)&= \dot{\alpha} \mathcal{T}_{4} \dfrac{h''(\alpha)}{h(\alpha)} \frac{\partial}{\partial w_i}\,,\quad 1\leq i\leq n-1\\
d^2\tau(\varphi)\left(\frac{\partial}{\partial w_k},\frac{\partial}{\partial \rho}\right)&= -\dot{\alpha} \tau_{\alpha} \dfrac{h''(\alpha)}{h(\alpha)} \frac{\partial}{\partial w_k}\,,\quad 1\leq k\leq n-1\,.
\end{split}
\end{equation}
Now, with respect to the frame \eqref{eq:coordinate-field}, the inner product of the two $2$-forms is
\[
\begin{split}
&\langle d^2\overline{\Delta}\tau(\varphi), d^2 \tau(\varphi) \rangle\\
&=\sum_{i,j=1}^{n}\sum_{k,\ell=1}^{n} (g_{M_f})^{ik}(g_{M_f})^{j\ell}
\langle d^2\overline{\Delta}\tau(\varphi)\left(\frac{\partial}{\partial w_i},\frac{\partial}{\partial w_j}\right), d^2\tau(\varphi)\left(\frac{\partial}{\partial w_k},\frac{\partial}{\partial w_\ell}\right)\rangle \,.
\end{split}
\]
Taking into account \eqref{eq:nonzeroinnerproductd2deltad2}, this reduces to

{\small
\begin{align}\label{expr-scalar-product-2-form-final}
&\langle d^2\overline{\Delta}\tau(\varphi), d^2 \tau(\varphi) \rangle\nonumber\\
=& 
\sum_{i=1}^{n-1}\sum_{k=1}^{n-1} (g_{M_f})^{ik}(g_{M_f})^{nn}
\langle d^2\overline{\Delta}\tau(\varphi)\left(\frac{\partial}{\partial w_i},\frac{\partial}{\partial \rho}\right), d^2\tau(\varphi)\left(\frac{\partial}{\partial w_k},\frac{\partial}{\partial \rho}\right)\rangle\nonumber\\
=& 
\sum_{i=1}^{n-1}\sum_{k=1}^{n-1} \frac{(g_{\s})^{ik}}{f^2} \left(\dot{\alpha} \mathcal{T}_4 \frac{h''(\alpha)}{h(\alpha)}\right) \left(-\dot{\alpha} \tau_{\alpha} \frac{h''(\alpha)}{h(\alpha)}\right) \langle \frac{\partial}{\partial w_i}, \frac{\partial}{\partial w_k}\rangle\nonumber\\
=& 
\sum_{i=1}^{n-1}\sum_{k=1}^{n-1} \frac{(g_{\s})^{ik}}{f^2} \left(\dot{\alpha} \mathcal{T}_4 \frac{h''(\alpha)}{h(\alpha)}\right) \left(-\dot{\alpha} \tau_{\alpha} \frac{h''(\alpha)}{h(\alpha)}\right) h^2(\alpha) (g_{\s})_{ki}\nonumber\\
=&
-(n-1) \dot{\alpha}^2 \tau_{\alpha} \mathcal{T}_4 \frac{h''^2(\alpha)}{f^2}\,.
\end{align}
}

It remains to study the behavior of the $3$-form $d^3 \tau(\varphi_{\alpha})$. For every $X, Y, Z \in C(TM_f)$,
{\small 
\begin{align}\label{expr-d3tau-general}
&d^3 \tau(\varphi_{\alpha})(X, Y, Z) \nonumber\\
 =& \nabla^{\varphi_{\alpha}}_X d^2 \tau(\varphi_{\alpha})(Y, Z) - \nabla^{\varphi_{\alpha}}_Y d^2 \tau(\varphi_{\alpha})(X, Z) + \nabla^{\varphi_{\alpha}}_Z d^2 \tau(\varphi_{\alpha})(X, Y)\nonumber \\
& - d^2 \tau(\varphi_{\alpha})\left([X, Y], Z\right) + d^2 \tau(\varphi_{\alpha})\left([X, Z], Y\right) - d^2 \tau(\varphi_{\alpha})\left([Y, Z], X\right).
\end{align}
}
Using \eqref{expr-d3tau-general}, it is a straightforward computation to check that $d^3\tau(\varphi_{\alpha})$ vanishes. Indeed, since 
$d^3\tau(\varphi_{\alpha})$ is a $3$-form, it is sufficient to apply \eqref{expr-d3tau-general} to triples of the type  
\[
\left(\frac{\partial}{\partial w_i}, \frac{\partial}{\partial w_j}, \frac{\partial}{\partial w_k} \right)\quad{\rm and}\quad\left(\frac{\partial}{\partial w_i}, \frac{\partial}{\partial w_j}, \frac{\partial}{\partial \rho} \right),\quad {\rm with} \quad 1 \leq i, j, k \leq n-1\,.
\] 

Finally,  
substituting \eqref{expr-codif-d2tau-normsquared-final} and \eqref{expr-scalar-product-2-form-final} in \eqref{ES-5-energia-generale}, we obtain the expression in the statement of the theorem.
\end{proof}

\vspace{3mm}
Now we are in the right position to prove Theorem~\ref{r-stability-theorem}.

\begin{proof}[\textbf{Proof of Theorem~\ref{r-stability-theorem}}] 
As we have recalled in Section~\ref{preliminaries}, a  weakly $r$-harmonic ($ES-r$-harmonic) map $\varphi_0 \in W^{r,2}\left ( B^n,\s^n \right )$ is  \textit{stable} if
for all variations $\varphi_s$ of $\varphi_0$ through maps in $W_{\varphi_0}^{r,2}\left ( B^n,\s^n \right )$ we have 
\begin{equation}\label{second-variation-r}
 \left . \frac{d^2}{ds^2}\,E_r\left( \varphi_s \right ) \right |_{s=0} \, \geq \, 0 \quad \left( \left . \frac{d^2}{ds^2}\,E_r^{ES}\left( \varphi_s \right ) \right |_{s=0} \, \geq \, 0\right)\,.
\end{equation}
We now prove that the three maps described in Theorems~\ref{th-constant-soltz} 
are unstable with respect to both the $r$-energy and the $ES$-$r$-energy. 
The three cases will be considered separately.

\noindent
\textbf{Case (i).} The map $\varphi_a:B^n \to \s^n$, defined as in \eqref{symmetricmaps}, is a proper weakly $ES-3$-harmonic ($3$-harmonic) map if and only if $n=7$ and 
\[
a=a_3=\dfrac{1}{2} \arccos\left(\dfrac{1}{9} \left(2\,\sqrt{10}-11\right)\right).
\]
Since the $3$-energy coincides with the $ES$-$3$-energy, 
it suffices to exhibit an admissible variation $\varphi_s$ of $\varphi_{a_3}$ such that
\[
\left . \frac{d^2}{ds^2}\,E_3\left( \varphi_s \right ) \right |_{s=0} \, < \, 0\,.
\]
Let $\alpha_s(\rho)=a_3+s (1-\rho)^3$. Then the corresponding map $\varphi_s\in W_{\varphi_0}^{3,2}\left ( B^n,\s^n \right )$ and, 
using the expression of $E_3$ given in \eqref{r-energy-rot-symm-maps-Lagrangian}, a computation carried out with the aid of  \textit{Mathematica}${}^{ \text{\textregistered}} $ gives
\[
\left . \frac{d^2}{ds^2}\,E_3\left( \varphi_s \right ) \right |_{s=0} \,=  \frac{2948 \sqrt{\frac{2}{5}}}{63}-\frac{12812}{315}\simeq -11.0781 < \, 0\,.
\]

\noindent
\textbf{Case (ii).}  The map $\varphi_a$ is a proper weakly $ES-4$-harmonic ($4$-harmonic) map if and only if $n=9$ and
\[
a=a_4=\dfrac{1}{2} \arccos\left(\dfrac{1}{16} \left(\sqrt{105}-19\right)\right) \,.
\]
In this case it suffices to prove the existence of an admissible variation $\varphi_s$ of $\varphi_{a_4}$ such that the second variations of both the $4$-energy and the $ES$-$4$-energy are strictly negative. Choosing $\alpha_s(\rho)=a_4+s (1-\rho)^4$, then the corresponding map satisfies $\varphi_s\in W_{\varphi_0}^{4,2}\left ( B^n,\s^n \right )$. Next, 
using the expression of $E_4$ given in \eqref{r-energy-rot-symm-maps-Lagrangian} and that of $E_4^{ES}$ recalled in \eqref{4-ES-energy-rot-symm-maps}, a computation again performed with the aid of \textit{Mathematica}${}^{ \text{\textregistered}} $ yields  
\[
\left . \frac{d^2}{ds^2}\,E_4\left( \varphi_s \right ) \right |_{s=0} \,=  \frac{9799 \sqrt{5}}{12\sqrt{21}}-\frac{43415}{84}\simeq -118.393 < \, 0\,;
\]
\[
\left . \frac{d^2}{ds^2}\,E_4^{ES}\left( \varphi_s \right ) \right |_{s=0} \,=  \frac{2546 \sqrt{5}}{3\sqrt{21}}-\frac{11090}{21}\simeq -113.988 < \, 0\,.
\]

\noindent
\textbf{Case (iii).} The map $\varphi_a$ is a proper weakly $ES-5$-harmonic ($5$-harmonic) map if and only if $n=11$ and
\[
a=a_{5}=\dfrac{1}{2} \arccos\left(\dfrac{1}{25} \left(6\sqrt{6}-29\right)\right) \,.
\]
First, we prove that there exists an admissible variation $\varphi_s $ of $\varphi_{a_5}$ such that the second variation of the $5$-energy is strictly negative. Indeed, let $\alpha_s(\rho)=a_5+s (1-\rho)^7$. Then the corresponding map $\varphi_s\in W_{\varphi_0}^{5,2}\left ( B^n,\s^n \right )$ and using the expression of $E_5$ given in \eqref{r-energy-rot-symm-maps-Lagrangian} we obtain
\[
\left . \frac{d^2}{ds^2}\,E_5\left( \varphi_s \right ) \right |_{s=0} \,=  \frac{48 \left(47124133 \sqrt{6}-116365497\right)}{446875}\simeq -100.476 < \, 0\,.
\]
As for the $ES-5$-energy, we consider the variation  $\alpha_s(\rho)=a_5+s (1-\rho)^{8}$. Then the corresponding map again satisfies $\varphi_s\in W_{\varphi_0}^{5,2}\left ( B^n,\s^n \right )$ and, using the expression of the $ES-5$-energy given in Theorem~\ref{th-Euler-Lag-r=5}, the usual computation yields
{\small 
\[
\begin{split}
\left . \frac{d^2}{ds^2} E_5^{ES}\left( \varphi_s \right ) \right |_{s=0} &=  \frac{128 \left(268481902 \sqrt{6}+60060 \sqrt{86362 \sqrt{6}-197208}-673907943\right)}{7596875}\\
&\simeq -152.878 < \, 0
\end{split}
\]
and so the proof is completed.
}
\end{proof}

\section{Proofs of Theorems~\ref{th-ellipsoid-biharm} and \ref{th-ellipsoid-triharm}}\label{sec-ellipsoid}

\begin{proof}[Proof of Theorem~\ref{th-ellipsoid-biharm}] First, using the argument of Lemma~\ref{lemma-belong-Sobolev-space}, we observe that $\varphi_a \in  W^{2,2}\left ( B^n,E^n(b) \right )$ if and only if $n \geq 5$.

Next, we observe that a map $\varphi_a:B^n \to E^n(b)$ as in \eqref{symmetricmaps-ellipsoid} belongs to the family of rotationally symmetric maps described by
\begin{equation}\label{symmetricmaps-ellipsoid-general}\begin{array}{lcll}
                                           \varphi_{\alpha} \,:\, &B^n &\to &  E^n(b) \subset \R^n \times \R \\
                                           &&& \\
                                           &x &\mapsto & ( \sin \alpha(\rho)\,\frac{x}{\rho}, b\,\cos \alpha(\rho)) \, .
                                         \end{array}
\end{equation}

Taking into account Proposition~\ref{pro:weakly-criterium}, we only have to check when $\varphi_{\alpha}$ restricted to $B^n\setminus\{O\}$ is biharmonic. A routine computation shows that the bienergy of these maps is given, up to an irrelevant constant factor, by
\begin{equation}\label{eq-bienergy-ellipsoid}
E_2\left( \varphi_\alpha \right )=\int_0^1 |\tau|^2 \rho^{n-1}\,d\rho \,,
\end{equation}
where
$\tau=\tau_\alpha \frac{\partial}{\partial \alpha}$ and 
\begin{equation}\label{eq-tau-Enb}
\tau_\alpha= \ddot{\alpha}+\frac{(n-1)}{\rho}\dot \alpha-\frac{(n-1)}{\rho^2} \frac{\sin \alpha \cos \alpha}{k^2(\alpha)}+\frac{k'(\alpha)}{k(\alpha)} \dot{\alpha}^2 \,,
\end{equation}
with 
$$k(\alpha)= \sqrt{\cos^2 \alpha + b^2 \sin^2 \alpha}\,.$$

Now, setting $L_2=|\tau|^2 \rho^{n-1}=\tau_\alpha^2\, k^2(\alpha)\, \rho^{n-1}$ and using \eqref{Euler-lagrange-generale} one can formally compute the biharmonicity equation. The full expression of the biharmonicity ODE is rather complicated but, in the special case that $\alpha(\rho)\equiv a$, the biharmonicity condition reduces to

\begin{equation}\label{eq-biharmonicity-condition-ellipsoid}
\begin{split}
&b^2 (n-4) \sin ^3(2 a)+2 b^2 \sin ^5 a \cos a \left(2 b^2 (n-4)-n+1\right)\\
&+6 (n-3) \sin a \cos ^5 a=0 \,.
\end{split}
\end{equation}

Since $0 < a < \pi/2$, setting $y=\tan ^2 a$ we easily find that \eqref{eq-biharmonicity-condition-ellipsoid} is equivalent to
\begin{equation}\label{eq-biharmonicity-condition-ellipsoid-y}
P_2(y)= y^2 \left(2 b^2 (n-4)-n+1\right)+4 (n-4) y+3 \frac{(n-3)}{b^2}=0 \,.
\end{equation}
Thus the conclusion of the theorem is equivalent to the statement that the second order polynomial $P_2(y)$ defined in \eqref{eq-biharmonicity-condition-ellipsoid-y} admits a \textit{positive} root if and only if \eqref{eq-cond-bihar-ellipsoid} holds. Now, since $n \geq 5$, this clearly happens if and only if $\left(2 b^2 (n-4)-n+1\right)<0$ which is equivalent to \eqref{eq-cond-bihar-ellipsoid} and so the proof is completed.
\end{proof} 
\begin{remark} The exact value of $a=a(b,n)$ which produces a weakly biharmonic map $\varphi_a:B^n \to E^n(b)$ can be computed explicitly solving the second order equation \eqref{eq-biharmonicity-condition-ellipsoid-y}. It is
\[
a=\arctan \left ( \sqrt{-\frac{\sqrt{-\left((n-1) \left(2 b^2 (n-4)-3 n+9\right)\right)}+2 b (n-4)}{2 b^3 (n-4)-b n+b}} \right ) \,.
\]
\end{remark}
\vspace{2mm}

\begin{proof}[Proof of Theorem~\ref{th-ellipsoid-triharm}]
First, using the argument of Lemma~\ref{lemma-belong-Sobolev-space}, we observe that $\varphi_a \in  W^{3,2}\left ( B^n,E^n(b) \right )$ if and only if $n \geq 7$.
Also in this case, we observe that a map $\varphi_a:B^n \to E^n(b)$ as in \eqref{symmetricmaps-ellipsoid} belongs to the family of rotationally symmetric maps described by \eqref{symmetricmaps-ellipsoid-general}. Moreover, by the argument of Proposition~\ref{pro:weakly-criterium}, it is sufficient to check when $\varphi_{\alpha}$ restricted to $B^n\setminus\{O\}$ is triharmonic. Now a routine computation shows that the trienergy of these maps is given, up to an irrelevant constant factor, by
\begin{equation}\label{eq-bienergy-ellipsoid}
E_3\left( \varphi_\alpha \right )=\int_0^1 L_3\,d\rho \,,
\end{equation}
where
\[
L_3= \rho^{n-1} \left((n-1) \frac{\tau^2_{\alpha}  \cos^2 \alpha}{\rho^2}+ {\dot\alpha}^2\left(\tau_{\alpha}k'(\alpha)+k(\alpha) \frac{\partial}{\partial\alpha}{{\tau}_{\alpha}}\right)^2\right)\,
\]
where $\tau_\alpha$ and $k(\alpha)$ where defined in \eqref{eq-tau-Enb}. Now, using \eqref{Euler-lagrange-generale}, one can formally compute the triharmonicity equation. Although the full expression of the triharmonicity ODE is rather complicated, in the special case that $\alpha(\rho)\equiv a$ the triharmonicity condition, in $x=\cos^2 a$, reduces to the polynomial equation
\begin{equation}\label{eq-triharmonicity-condition-ellipsoid}
P_{b}(x)=A_3 x^3+A_2 x^2+A_1 x+A_0=0 \,,
\end{equation}
where
{\small
\[
\begin{split}
A_3=&(b^2-1)  \left(2(n-4) b^2-3 (n-3)\right) \left(4 (n-6) b^2-5 (n-5)\right)\\
A_2=&-b^2 \left(24 (n-4) (n-6)b^4 + (-62 n^2+566 n-1224)b^2+41 n^2-332 n+651\right)\\
A_1=&2 b^2 \left(12 (n-4) (n-6) b^4+(-17 n^2+149 n-312)b^2+(n-1) (2 n-5)\right)\\
A_0=&-2 b^4 (n-4) \left(4(n-6) b^2-n+1\right)\,.
\end{split}
\]
}
Thus the conclusion of the theorem is equivalent to the statement that the third order polynomial $P_b(x)$ defined in \eqref{eq-triharmonicity-condition-ellipsoid} admits a root in the interval $(0,1)$  if and only if \eqref{eq-cond-trihar-ellipsoid} holds. Now, since $n \geq 7$, $P_b(1)=-15 \left(n^2-8 n+15\right)<0$. Moreover, if \eqref{eq-cond-trihar-ellipsoid} holds, then $P_b(0)>0$ and so there exists a root of $P_b(x)$
in the interval $(0,1)$.

Conversely, let us assume that \eqref{eq-cond-trihar-ellipsoid} does not hold, that is
\[
b^2\geq\frac{(n-1)}{4(n-6)}=b_*^2\,.
\]
We prove that in this case $P_{b}(x)$ is strictly negative for $x \in(0,1)$. To this purpose, first we show that $P_{b_*}(x)<0$ for all $x\in (0,1)$ and $n\geq 7$. Indeed, we have 
\begin{eqnarray*}
P_{b_*}(x)&=& \frac{x}{4(n-6)}\Big [-3 (n-5) (n-1)^2-(n-1) \left(27 n^2-268 n+601\right) x\\
&&-2 (3 n-23) \left(5 n^2-49 n+104\right) x^2 \Big ]\,.
\end{eqnarray*}
Now, direct inspection shows that $P_{b_*}(x)<0$ for $x\in (0,1)$ and $n\geq 8$. If $n=7$, $P_{b_*}(x)$ becomes
\[
\frac{x}{4}(24 x^2-288 x-216)
\]
which is negative for all $x\in (0,1)$. Now, in summary, to end the proof it is sufficient to show that for $b^2> b_*^2$ we have
\[
P_{b_*}(x)-P_{b}(x)\geq 0\,.
\]
A direct computation yields
\[
P_{b_*}(x)-P_{b}(x)=\frac{(1-x) (4 (n-6)b^2+(1-n))}{4 (n-6)} Q_b(x)\,,
\]
where
\begin{eqnarray*}
Q_b(x)&=&x \left(30 n^2 x+3 n^2-286 n x-18 n+616 x+15\right)\\
&&+4 (n-6) (7 n-25) (1-x)\, x\, b^2+8 (n-6) (n-4) (1-x)^2 b^4\,.
\end{eqnarray*}

Thus it remains to show that $Q_b(x)\geq 0$ for $x\in (0,1)$, $n\geq 7$ and $b^2> b_*^2$.

Indeed, since $n\geq 7$, $b^2> b_*^2$ implies $Q_b(x)>Q_{b_*}(x)$, where
\[
\begin{split}
Q_{b_*}(x)=&\frac{1}{2(n-6)}\Big[((n-4) (n-1)^2+2 (n-4) (n-1) (9 n-59) x\\
&+\left(47 n^3-790 n^2+4239 n-7096\right) x^2\Big]\,.
\end{split}
\]
Now, direct inspection shows that $Q_{b_*}(x)>0$ when $n\geq 8$ and $x\in (0,1)$. Finally, if $n=7$, then $Q_{b_*}(x)$ becomes
\[
\frac{-12 x^2+144 x+108}{4}
\]
which is clearly positive for all $x\in (0,1)$ and so the proof is completed.
\end{proof}

\section{Replacing $B^n$ by a warped product}\label{sec-domain-not-Bn}
In Section~\ref{sec-ellipsoid} we proved that certain proper weakly biharmonic and triharmonic maps from $B^n$ onto a parallel hypersphere do exist when the target sphere $\s^n$ is replaced by a suitable ellipsoid $E^n(b)$. More precisely, differently from the case of a spherical target, a convenient choice of $b$ ensures existence also when the dimension $n$ is arbitrarily large.

In this section we analyse what happens if the domain, i.e., the Euclidean ball $B^n$, is replaced by a warped product manifold as in \eqref{rotationallysymmetricmaps-models}. We assume $h(\alpha)=\sin \alpha$ and so the target is isometrically equivalent to $\s^n$. 

More specifically, as a special case within the family of rotationally symmetric maps \eqref{rotationallysymmetricmaps-models}, here we study maps of the following form:
{\small
\begin{equation}\label{rotationallysymmetricmaps-model-to-sphere}
\begin{array}{rll}
                                           \varphi_{a} \colon\left (\s^{n-1} \times I, f^2(\rho)g_{\s}+d\rho^2 \right ) &\to &  \left (\s^{n-1} \times [0,\pi], \sin^2\!\alpha\, g_{\s}+d\alpha^2 \right )=\s^n \\
                                           && \\
                                           (w,\rho)& \mapsto & (w,a) \,\, , 
                                         \end{array}
\end{equation}
}
where $a$ is a constant such that $0<a<\pi/2$. It is natural to focus on the case that $\overline{I}=[0,1]$ and $f$ verifies in the pole the smoothness assumptions \eqref{condizioni-su-f}. In this case the domain is a unit geodesic ball centered at the pole which we denote by $\mathcal{B}^n_f$.
Our first result in this context is the following:
\begin{theorem}\label{th-model-non-existence}
Assume $n \geq 5$. Let $\mathcal{B}^n_f$ be a unit geodesic ball and $\varphi_a:\mathcal{B}^n_f \to \s^n$ a map as in \eqref{rotationallysymmetricmaps-model-to-sphere}. 
If $\varphi_a$ is weakly biharmonic, then $n=5$ or $n=6$, $\mathcal{B}^n_f$ is the Euclidean ball $B^n$ and $\varphi_a$ is one of the maps given in Theorem~\ref{theorem-m=5-6biharmonic}.
\end{theorem}
\begin{proof} 
As in Lemma~\ref{lemma-belong-Sobolev-space}, a map of the type \eqref{rotationallysymmetricmaps-model-to-sphere} belongs to $W^{2,2}\left (\mathcal{B}^n_f , \s^n \right ) $ if and only if $n \geq 5$. Moreover, a straightforward computation shows that for a rotationally symmetric map $\varphi_a$ as in \eqref{rotationallysymmetricmaps-model-to-sphere} the biharmonicity equation is equivalent to
\begin{equation}\label{eq-biharm-from-models}
(n-1) \cos (2 a)+2 f(\rho ) f''(\rho )+2 (n-4) f'(\rho )^2=0\,.
\end{equation}

First, inspection of \eqref{eq-biharm-from-models} in $\rho=0$ using \eqref{condizioni-su-f} tells us that
\[
\cos (2a)= \frac{2(n-4)}{(1-n)} \,.
\]
It follows that necessarily $n=5$ or $n=6$, and \eqref{eq-biharm-from-models} can be rewritten as follows:
\begin{equation}\label{eq-ODE-f}
f(\rho ) f''(\rho )+ (n-4) f'(\rho )^2=(n-4)\,.
\end{equation}
First, let us assume that $n=5$. Then \eqref{eq-ODE-f} becomes
\[
f(\rho ) f''(\rho )+  f'(\rho )^2=1
\]
which can be explicitly integrated. Indeed, a first integration gives
\[
f(\rho)f' (\rho)= \rho +c_1
\]
and then
\[
f(\rho)=\sqrt{\rho^2 + 2\,c_1\,\rho + c_2} \,.
\]
But now imposing $f(0)=0$ and $f'(0)=1$ we easily conclude that $c_2=c_1=0$ and so $f(\rho)=\rho$ as requested. Then we are in the context of Theorem~\ref{th-model-non-existence} and so the proof is completed in this case. 
\vspace{2mm}

Now, let us assume that $n = 6$. Then the biharmonicity ODE becomes
\begin{equation}\label{ODE-1}
f f'' + 2 f'^2 - 2 = 0 \,.
\end{equation}
Multiplying \eqref{ODE-1} by $2 f^3 f'$ we obtain
\[
2 f^4 f' f'' + 4 f^3 f'^3 - 4 f^3 f' = 0\,,
\]
that is
\[
(f^4 f'^2)'-(f^4)'=0\,.
\]
Integrating we get
\[
f^4 f'^2=f^4 +C\,.
\]
Finally, imposing $f(0)=0$ we deduce that $C=0$ and so it is easy to conclude that $f(\rho)=\rho$. Then we are in the context of Theorem~\ref{th-model-non-existence} and so the proof is completed in this case as well.
\end{proof}
\vspace{3mm}
Next, we focus on the triharmonic case. Our result is:
\begin{theorem}\label{th-model-non-existence-tri}
Assume $n \geq 7$. Let $\mathcal{B}^n_f$ be a unit geodesic ball and $\varphi_a:\mathcal{B}^n_f \to \s^n$ a map as in \eqref{rotationallysymmetricmaps-model-to-sphere}.
\begin{itemize}
\item[(i)] If $\varphi_a$ is weakly triharmonic, then $n=7$.
\item[(ii)] If $n=7$, $f(\rho)$ is a real analytic function in a neighbourhood of $\rho=0$ and $\varphi_a$ is weakly triharmonic, then $\mathcal{B}^n_f$ is the Euclidean ball $B^n$ and $\varphi_a$ is the map given in Theorem~\ref{th-constant-soltz}(i).
\end{itemize}
\end{theorem}
\begin{proof} As an application of Theorem~\ref{theorem-r-energy} we find that the condition of triharmonicity for $\varphi_a:\mathcal{B}^n_f \to \s^n$ is equivalent to:
{\small
\begin{eqnarray}\label{tri}
&&\\\nonumber
&& 4 (n-1) (2 \cos (2 a)+1) f(\rho) f''(\rho)\\\nonumber
&&+4 f'(\rho)^2 \Big [(n-1) (2 (n-5) \cos (2 a)+n-6)+\left(-2 n^2+33 n-103\right) f(\rho) f''(\rho)\Big]\\\nonumber
&&+(n-1)^2 \cos ^2(a) (3 \cos (2 a)-1)-4 f(\rho)^3 f^{(4)}(\rho)+4 (11-2 n) f(\rho)^2 f''(\rho)^2\\\nonumber
&&+16 \left(n^2-10 n+24\right) f'(\rho)^4-12 (n-5) f(\rho)^2 f^{(3)}(\rho) f'(\rho)=0\,.
\end{eqnarray}
}
Next, imposing the initial conditions
\[
f(0)=0\,, \quad f(0)=1\,,\quad f''(0)=0
\]
we obtain
\[
8 \left(n^2-6 n+5\right) \cos (2 a)+(n-1)^2 \cos ^2a \,(3 \cos (2 a)-1)+20 n^2-188 n+408=0\,.
\]
Now, a simple analysis as in the proof of Theorem~\ref{th-constant-soltz} shows that necessarily $n=7$ and 
$$a=a_3=\dfrac{1}{2} \arccos\left(\dfrac{1}{9} \left(2\,\sqrt{10}-11\right)\right)\,.
$$ 
Then the proof of part (i) of the theorem is completed. As for statement (ii), we first observe that inserting explicitly the values $n=7$ and $a=a_3$ into \eqref{tri} the triharmonicity condition becomes equivalent to
\begin{eqnarray}\label{eq:main}
&&-3 f^3 f^{(4)}
-9 (f f'')^2
+(8\sqrt{10}-26) f f''
+36 (f')^4
\\\nonumber
&&-18 f^2 f^{(3)} f'
+2 (f')^2 (45 f f'' + 8\sqrt{10}-35)
-16\sqrt{10}+34
=0\,.
\end{eqnarray}
The assumption that $f$ is real analytic in a neighbourhood of $\rho=0$ together with \eqref{condizioni-su-f} tell us that $f(\rho)$ can be written as a power series as follows:
\begin{equation}\label{eq-f-real-analytic}
f(\rho)=\rho+\sum_{j=1}^{+\infty}b_{2j+1}\,\rho^{2j+1}\,.
\end{equation}
We prove that $f(\rho)=\rho$ by induction. Indeed,

\textbf{Step 1, $b_3=0$.} First, using \eqref{eq-f-real-analytic}, we observe that we can write:
\begin{equation}\label{eq-espressioni-ff''-f'2}
\begin{array}{lll}
f^3 f^{(4)}=\sum_{j=1}^{+\infty} c_j\,\rho^{2j+2} &
f\,f''=6 \, b_3 \rho^2+\sum_{j=1}^{+\infty} d_j\,\rho^{2j+2}\\[8pt]
(f\,f'')^2=\sum_{j=1}^{+\infty} e_j\,\rho^{2j+2} &
f'^2=1+6 \,b_3 \rho^2 +\sum_{j=1}^{+\infty} f_j\,\rho^{2j+2}\\[8pt]
f'^4=1+12 \, b_3 \rho^2 +\sum_{j=1}^{+\infty} g_j\,\rho^{2j+2}&
f^2 f^{(3)}f'=6 \, b_3 \rho^2+\sum_{j=1}^{+\infty} h_j\,\rho^{2j+2}\\[8pt]
f'^2 f f''=6 \, b_3 \rho^2+\sum_{j=1}^{+\infty} k_j\,\rho^{2j+2}
\end{array}
\end{equation}

where the explicit expressions of the coefficients $c_j,\dots, k_j$ play no role.
Now, replacing \eqref{eq-espressioni-ff''-f'2} into \eqref{eq-ODE-f}, we obtain that the vanishing of the coefficient of $\rho^2$ in the power series reduces to
\[
144 \left(\sqrt{10}+2\right) b_3=0
\]
which implies $b_3=0$.

\textbf{Step 2, induction.} We assume $b_{2j+1}=0$ for $1 \leq j \leq j^*$. Then we have to prove that $b_{2(j^*+1)+1}=0$, i.e., $b_{2j^*+3}=0$.

Now, under the inductive hypothesis, the version of \eqref{eq-espressioni-ff''-f'2} becomes
\begin{equation}\label{eq-espressioni-ff''-f'2-induction}
\begin{split}
f^3 f^{(4)}&=(2j^*+3)(2j^*+2)(2 j^*+1)(2 j^*)b_{2j^*+3}\, \rho^{2j^*+2}+\sum_{j=j^*}^{+\infty} c_j\,\rho^{2j+4}\, ;\\
f\,f''&=(2j^*+3)(2j^*+2)b_{2j^*+3}\, \rho^{2j^*+2}+\sum_{j=j^*}^{+\infty} d_j\,\rho^{2j+4}\, ;\\
(f\,f'')^2&=\sum_{j=2j^*}^{+\infty} e_j\,\rho^{2j+4}\, ;\\
f'^2&=1+2(2j^*+3)b_{2j^*+3}\, \rho^{2j^*+2} +\sum_{j=j^*}^{+\infty} f_j\,\rho^{2j+4}\,;\\
f'^4&=1+4(2j^*+3)b_{2j^*+3}\, \rho^{2j^*+2}+\sum_{j=j^*}^{+\infty} g_j\,\rho^{2j+4}\,;\\
f^2 f^{(3)}f'&=(2j^*+3)(2j^*+2)(2 j^*+1)b_{2j^*+3}\, \rho^{2j^*+2}+\sum_{j=j^*}^{+\infty} h_j\,\rho^{2j+4}\, ;\\
f'^2 f f''&=(2j^*+3)(2j^*+2)b_{2j^*+3}\, \rho^{2j^*+2}+\sum_{j=j^*}^{+\infty} k_j\,\rho^{2j+4}\\
\end{split}
\end{equation}
where again the explicit expressions of the coefficients $c_j,\dots, k_j$ play no role.
Next, using \eqref{eq-espressioni-ff''-f'2-induction} into \eqref{eq-ODE-f}, we find that the lower order nonzero term in the power series is $\rho^{2j^*+2}$ and its coefficient is
\[
4 \, b_{2j^*+3} (2 j^*+3) \left(-6 j^*{^3}-27 j^*{^2}+\left(4 \sqrt{10}+2\right) j^*+12(2+\sqrt{10})\right)\,.
\]
Now, since the third order polynomial 
\[
-6 j^*{^3}-27 j^*{^2}+\left(4 \sqrt{10}+2\right) j^*+12(2+\sqrt{10})
\]
admits no integer root, we deduce that the vanishing of the coefficient of $\rho^{2j^*+2}$ implies $b_{2j^*+3}=0$ and so the proof is completed.
\end{proof}

\newpage

{\bf Statements and Declarations}\\

{\em Funding}: The authors are members of the Italian National Group G.N.S.A.G.A. of INdAM. The work was partially supported by the Project {PRoBIKI} funded by Fondazione di Sardegna.  The author A.S. was supported by a NRRP scholarship - funded by the European Union - NextGenerationEU - Mission 4, Component 1, 
Investment 3.4.\\

{\em Competing Interests}: The authors have no relevant financial or non-financial interests to disclose.

\end{document}